\documentclass[journal,onecolumn]{IEEEtran}

\usepackage{nicematrix}

\usepackage{amsmath} 
\usepackage{amssymb}  
\usepackage{amsthm} 
\usepackage{amsfonts}
\usepackage{graphicx}
\usepackage{xcolor}
\usepackage{mathtools}
\usepackage{subcaption}
\usepackage{enumitem}
\usepackage{lipsum}
\usepackage{mathtools}
\usepackage{cuted}
\usepackage{cite}
\usepackage{chemformula}

\usepackage[hidelinks]{hyperref}

\usepackage{mathrsfs} 

\newtheorem{theorem}{Theorem}
\newtheorem{lemma}{Lemma}
\newtheorem{assumption}{Assumption}

\newtheorem{proposition}{Proposition}

\usepackage{pgf,tikz}
\usepackage[utf8]{inputenc}
\usepackage{pgfplots} 
\usepackage{pgfgantt}
\usepackage{pdflscape}
\pgfplotsset{compat=newest} 
\pgfplotsset{plot coordinates/math parser=false}

\usepackage{verbatim}
\begin{document}
\title{\LARGE \bf
Modeling the Co-evolution of Climate Impact and Population Behavior: A Mean-Field Analysis
}

\author{Kathinka Frieswijk, Lorenzo Zino, A. Stephen Morse, and Ming Cao
\thanks{K. Frieswijk and M. Cao are with the Faculty of Science and Engineering, University of Groningen, Groningen, the Netherlands (\texttt{\{k.frieswijk,m.cao\}@rug.nl}). L. Zino is with the Department of Electronics and Telecommunications, Politecnico di Torino, Turin, Italy  (\texttt{lorenzo.zino@polito.it}). A.S. Morse is with the Department of Electrical Engineering, Yale University, New Haven, CT 06511, USA (\texttt{as.morse@yale.edu}). This work was partially supported by the European Research Council (ERC-CoG-771687) and by the US Air Force (grant n. FA9550-16-1-0290). 
}%
}

\maketitle

\begin{abstract}
\noindent Motivated by the climate crisis that is currently ravaging the planet, we propose and analyze a novel framework for the evolution of anthropogenic climate impact in which the evolution of human environmental behavior and environmental impact is coupled. Our framework includes a human decision-making process that captures social influence, government policy interventions, and the cost of environmentally-friendly behavior, modeled within a game-theoretic paradigm. By taking a mean-field approach in the limit of large populations, we derive the equilibria and their local stability characteristics. Subsequently, we study global convergence, whereby we show that the system converges to a periodic solution for almost all initial conditions. Numerical simulations confirm our findings and suggest that, before the system reaches such a periodic solution, the level of environmental impact might become dangerously high, calling for the design of optimal control strategies to influence the system trajectory. 
\end{abstract}


\section{Introduction}\label{sec:intro}

\noindent \IEEEPARstart{W}{ith} the rapid rise in temperature and extreme weather conditions registered in the last few years all around the globe, it is hard to deny that climate change is a serious threat to all life on our planet. In a myriad of ecosystems, the climate change crisis has already been the cause of substantial damages, and often irreversible losses of biodiversity~\cite{portner2022climate}. 

To mitigate the consequences of the climate crisis, a collective adoption of environmentally-responsible behavior is necessary~\cite{otto2020social}. However, even though there is an increasing global awareness that the climate crisis is real, dangerous, and occurring right now, such awareness has not yet translated into sufficiently resolute actions, able to decrease carbon dioxide emissions. On the contrary, preliminary data for the year 2022 suggest a relative increase in global fossil \ch{CO2} emission of $1.0\%$ compared to 2021~\cite{friedlingstein2022global}, thereby reaching an atmospheric \ch{CO2} concentration of 417.2 ppm, which is 51\% above pre-industrial levels of around 278 ppm.  

Toward predicting the collective adoption of sustainable practices, it is key to develop accurate models of the individual-level mechanisms that drive people to make decisions on their behavior. Evolutionary game theory has emerged as a powerful framework to develop such models~\cite{hofbauer1998evolutionary}. Of particular interest are feedback-evolving games in which the behavior of individuals influences the surrounding environment, while the environment, in turn, impacts the decision-making process~\cite{GONG2022110536,weitz2016oscillating}. 

Feedback-evolving games have proved useful in explaining dynamical phenomena in biological systems, such as resource harvesting and plant nutrient acquisition~\cite{tilman2020evolutionary}. However, such models do inherently oversimplify the complex and evolving nature of human behavior by assuming that individual decision-making is governed by a game whose payoff matrix depends linearly on the surrounding environment. Hence, such frameworks are
not amenable to the inclusion of nonlinear features due to the role of social influence, and furthermore, they do not explicitly consider the (potentially time-varying) implementation of policy interventions. Therefore, feedback-evolving games are limited in their practical applicability.

To address this gap, we propose a novel mathematical network model for the co-evolution of anthropogenic environmental impact and human behavior, where the decision-making process of individuals includes factors such as social influence, policy interventions, and the cost of environmentally responsible behavior. We propose a behavioral revision process in which individuals have a tendency to imitate individuals with a higher payoff~\cite{hofbauer1998evolutionary}, while they 
also prefer to conform to the behavioral norm of their social environment~\cite{cialdini2004social}. We formulate our model as a continuous-time Markov process. By taking a mean-field approach in the limit of large populations~\cite{virusspread}, we derive a deterministic approximation of our stochastic model in the form of a system of two coupled nonlinear ordinary differential equations. Then, we perform a theoretical analysis of the obtained system. In particular, we start by deriving local stability properties of the system equilibria. Subsequently, we utilize these results and an argument based on the Poincaré-Bendixson theorem in order to study the global asymptotic behavior of the system. Specifically, we prove that for (almost) every initial condition in the interior of the domain, the system converges to a periodic solution.  Numerical simulations are provided to illustrate our findings and to explore feedback control policies to mitigate risky oscillations in the transient behavior of the system.

The rest of this paper has the following organization. After presenting some notation below, we introduce our modeling framework in Section~\ref{sectionmodel}. In Section~\ref{sectionmf}, we take a mean-field approach in the limit or large populations and derive the mean-field system dynamics. Subsequently, in Section~\ref{sectionresults}, we present the analysis of the mean-field model and our main theoretical results. The paper is concluded by Section~\ref{sec:con}, which discusses future research avenues. 

\noindent \emph{Notation:} Let $\mathbb{R}$, $\mathbb{R}_{\ge 0}$, and $\mathbb{R}_{> 0}$ denote the set of real, real nonnegative, and strictly positive real numbers, respectively. If, for an event $E$, $$\lim_{\Delta t \searrow 0} \dfrac{{\mathbb{P}\big[E\text{ occurs during }(t,t+\Delta t)\big]}}{\Delta t} =\rho_E (t),$$ then we state that $E$ is triggered by a \emph{Poisson clock} with rate $\rho_E(t)$.

\section{Model}\label{sectionmodel}

\noindent We consider a population of $n$ individuals, denoted by $\mathcal{V}:=\{1,\hdots,n\}$. Each individual is represented by a vertex in a directed network $\mathcal{G}(t) := (\mathcal{V}, \mathcal{E})$, where $(i,j) \in \mathcal{E}$ if and only if (iff) $j$ has a social influence on the behavior of $i$. The neighbor set of $i$ is denoted by $\mathcal N_i : = \{j \in \mathcal{V} \ : \ (i,j) \in \mathcal{E} \},$ with cardinality $d_i:=|\mathcal N_i|$. The environmental behavior of an individual $i \in \mathcal{V}$ at time $t \in \mathbb{R}_{\ge 0}$ is captured by $x_i (t) \in \{0,1\}$, which represents whether $i$ is displaying environmentally responsible behavior $(x_i (t) = 1)$, or environmentally irresponsible behavior $(x_i (t) = 0)$. The states of all individuals are gathered into an $n$-dimensional vector $X(t):=\left[ x_1(t), x_2(t) , \dots, x_n(t) \right]\in\{ 0,1\}^{n}$, which represents the behavior of the entire population at time $t$.

\subsection{Environmental Impact}
\noindent For the past 50 years, anthropogenic $\ch{CO2}$ (i.e., the increase in the atmospheric value of $\ch{CO2}$ with respect to the pre-industrial value) has increased exponentially~\cite{hofmann2009new,friedlingstein2022global}. Therefore, we choose to model the evolution of the anthropogenic environmental impact $\varepsilon \in \mathbb{R}_{\ge 0}$ through the following linear, non-autonomous ordinary differential equation (ODE):
\begin{equation}\label{environmentimpact}
  \dot{\varepsilon} = r(t) \varepsilon, 
\end{equation}
where the rate of growth or decay at time $t$ is given by 
\begin{equation}\label{eq:rate}
r(t) : = \gamma\bar{x}_0 (t)  - \tau.  
\end{equation}
Here, the effect of environmentally irresponsible behavior is modeled by $\gamma\bar{x}_0 (t)$, with $\gamma \in \mathbb{R}_{>0}$ and where
\begin{equation}
  \bar{x}_0 (t):= \tfrac{1}{n}\big| \{i \in \mathcal{V} \ : \ x_i(t) = 0 \} \big|  
\end{equation}
denotes the fraction of people who behave irresponsibly in the population at time $t$. The parameter $\tau\in \mathbb{R}_{>0}$ represents efforts to reduce environmental impact via, e.g., massive tree-planting projects or negative emissions technologies. 


\subsection{Environmental Behavior}
\noindent Inspired by the decision-making process that was proposed in~\cite{ye2021game, frieswijk2022mean}, each individual $i\in \mathcal{V}$ chooses whether to behave environmentally responsibly following an evolutionary game-theoretic mechanism~\cite{hofbauer1998evolutionary}, which depends on the environmental impact, social influence, the cost of environmentally-friendly behavior, and governmental policy interventions such as awareness campaigns and environmental subsidies. In particular, for any $i \in \mathcal{V}$, we define the \textit{incentive} $\iota_1^{(i)}$ for environmentally responsible behavior as
\begin{subequations}\label{incentive}
 \begin{align}
    \iota_1^{(i)} (X(t),\varepsilon(t)) &:=\displaystyle\frac{1}{d_i} \sum_{j\in\mathcal N_i} x_j(t) + \mu \varepsilon(t) + \alpha  \,,\label{incentive1}\\
    \intertext{whereas} 
  \iota_0^{(i)} (X(t))&:=\displaystyle\frac{1}{d_i} \sum_{j\in\mathcal N_i} \big(1-x_j(t)\big) + \kappa - \sigma \,,\label{incentive0}
 \end{align}
\end{subequations}
denotes the incentive for environmentally irresponsible behavior. The behavioral incentives include several terms, whose meaning is detailed in the following.
\begin{description}
 \item[Social influence.] The first element in \eqref{incentive1}-\eqref{incentive0} represents social influence, where the incentive of $i \in \mathcal{V}$ to display certain behavior is higher when more of $i$'s neighbors act accordingly. This reflects the tendency of individuals to conform to their social environment~\cite{cialdini2004social}. Field experiments showed that social norms influence the behavior of individuals, e.g., curbside recycling behavior~\cite{schultz1999changing}. 
 \item[Environmental response.] The term $\mu \varepsilon(t) $, with $\mu \in \mathbb{R}_{> 0}$, models the response of the population to the environmental impact $\varepsilon (t)$, which is reflected in, e.g., global food price inflation and shortages~\cite{portner2022climate}. Here, we assume that the population's response increases linearly with the impact on the environment, regulated by the parameter $\mu$: the larger $\mu$ is, the faster the population reacts to environmental changes. However, one may consider more complex and nonlinear response functions, similar to~\cite{ye2021game}.
 \item[Cost.] The higher costs of behaving in an environmentally favorable fashion represent a barrier to ``green"  behavior of individuals~\cite{young2010sustainable}. Thus, the incentive for irresponsible behavior is bolstered by the cost of acting responsibly, represented by $\kappa \in \mathbb{R}_{>0}$. Such a term captures the (direct and indirect) economical costs associated with responsible behavior. 
 \item[Environmental subsidies.] Environmentally-friendly behavior is stimulated by government subsidies, modeled by reducing the cost for responsible behavior $\kappa$ by $\sigma \in [0,\kappa]$.
 \item[Awareness campaigns.]
 Besides the cost, another barrier to ecologically sustainable behavior is a lack of available information on how to act in a responsible fashion~\cite{young2010sustainable}. Awareness campaigns, modeled by the parameter $\alpha \in \mathbb{R}_{\ge 0}$, boost public knowledge and thereby increase the incentive to behave responsibly.
\end{description}
Individuals undergo a behavioral change according to a stochastic adaptation of classical \textit{imitation dynamics}, often employed in evolutionary game theory~\cite{hofbauer1998evolutionary,como2020imitation}. In particular, an individual $i \in \mathcal{V}$ who is behaving irresponsibly at time $t$ (i.e., $x_i (t)= 0$) will adopt responsible behavior if triggered by a Poisson clock with rate
\begin{subequations}\label{behaviourevolution}
 \begin{align}
   \rho_{01}^{(i)}(X(t),\varepsilon(t))&=\frac{1}{d_i}\sum_{j\in \mathcal N_i}x_j(t)\iota_1^{(j)}(X(t),\varepsilon(t))\,,\label{behaviourevolution_01}\\
 \intertext{whereas an individual $i \in \mathcal{V}$ who is displaying responsible behavior (i.e., $x_i(t)=1$) will cease to do so if triggered by a Poisson clock with rate} 
\rho_{10}^{(i)}(X(t))&=\frac{1}{d_i}\sum_{j\in \mathcal N_i}\big(1-x_j(t)\big)\iota_0^{(j)}(X(t))\,.\label{behaviourevolution_10}
 \end{align}
\end{subequations}
The revision protocol driven by the rates in \eqref{behaviourevolution} has an intuitive interpretation. Individuals interact with their neighbors and revise their own behavior imitating their neighbors with a probability proportional to the incentive associated with the corresponding behavior, similar to classical imitation dynamics~\cite{hofbauer1998evolutionary}. The proposed \textit{conformity-driven imitation dynamics} combines the incentive-driven behavioral tendencies of individuals with their propensity to conform to the behavioral norm of their social environment~\cite{cialdini2004social}. The above state transitions for an individual $i\in \mathcal{V}$ are shown in Fig.~\ref{fig:dynamics}.

To summarize, the behavioral-environmental feedback model is characterized by the coupling between i) a shared environment $\varepsilon(t)\in\mathbb R_{\geq 0}$, whose evolution is captured by the ODE in \eqref{environmentimpact} with the rate of growth/decay in \eqref{eq:rate}, and ii) the behavior $X(t)$ of a network of $n$ individuals, which is updated according to the revision protocol in \eqref{behaviourevolution}, with incentives from \eqref{incentive}.


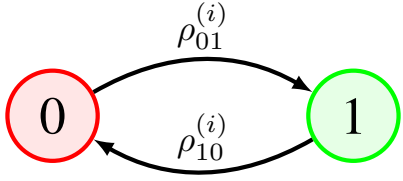
\begin{figure}
  \centering
  \begin{tikzpicture}
\node[draw=green, fill=green!10,circle, ultra thick,,minimum size=1.2cm] (1) at (4,-.5) {\huge 1};
\node[draw=red, fill=red!10,circle, ultra thick,,minimum size=1.2cm] (0) at (0,-.50) {\huge 0};
\path [->,>=latex,ultra thick] (0) edge[bend left =30]  node [above] {\scalebox{1.35}{$\rho_{01}^{(i)}$}} (1);
\path [->,>=latex,ultra thick] (1) edge[bend left =30]  node [above] {\scalebox{1.35}{$\rho_{10}^{(i)}$}} (0);
\end{tikzpicture}
  \caption{The transition rates between the behavioral states $x_i = 0$ and $x_i=1$ for a generic individual $i \in \mathcal{V}$.}
  \label{fig:dynamics}
\end{figure}

\section{Mean-Field Dynamics}\label{sectionmf}

\noindent All of the Poisson clocks associated with the individuals' behavioral transitions are independent. Hence, the population's behavioral state $X(t)\in\{0,1\}^n$ evolves according to a non-homogeneous continuous-time Markov process~\cite{levin2006book}. Specifically, for any $i \in \mathcal{V}$, the transition rate matrix is given by 
\begin{equation}\label{eq:Q}
Q_i(X(t),\varepsilon(t))=\begin{bmatrix}
-\rho_{01}^{(i)}(X(t),\varepsilon(t)) & \rho_{01}^{(i)}(X(t),\varepsilon(t)) \\
\rho_{10}^{(i)}(X(t)) & - \rho_{10}^{(i)}(X(t))
\end{bmatrix},
\end{equation}
where the first and second row/column correspond to the state $x_i =0$ and $x_i =1$, respectively. Thus, the probability that any $i \in \mathcal{V}$ transitions from behavior $y\in\{0,1\}$ to $z\in\{0,1\}$ at time $t$, with $y\neq z$, is given by
$$
\mathbb P[x_i(t+\Delta t)=z\,|\,x_i(t)=y]=\big(Q_i(X(t),\varepsilon(t))\big)_{yz}\Delta t+o(\Delta t)\,,
$$
where  $o(\Delta t)$ is the Landau little-o notation for $\Delta t \searrow 0$. 

From the explicit expression of the transition rate matrix $Q_i(X(t),\varepsilon(t))$, we realize that all its entries are dependent on the behavior of the other population members (i.e., $X(t)$) through the dependency on the state of the neighboring nodes in \eqref{behaviourevolution} and \eqref{incentive}. Moreover, the transition rates are non-homogeneous, since the first row is also dependent on $\varepsilon(t)$. The complexity of the transition matrix $Q_i(X(t),\varepsilon(t))$ and the fact that the size of the state space $\{0,1\}^n$ increases exponentially with the population size $n$ make a direct analysis of the non-homogeneous Markov process $X(t)$ unfeasible for large-scale populations. 

Following the seminal work in~\cite{virusspread}, we take a mean-field approach in the limit $n \to \infty$. In this approach, instead of studying the actual evolution of the state of each individual $x_i(t)$, we study the probability for any $i \in \mathcal{V}$ to act irresponsibly and responsibly, defined as $p_0^{(i)}(t) : = \mathbb{P}\left[ x_i(t) = 0 \right]$ and $p_1^{(i)}(t) : = \mathbb{P}\left[ x_i(t) = 1 \right]$, respectively. For any $i\in\mathcal{V}$, the evolution of the mean-field dynamics for $p_0^{(i)}(t)$ and $p_1^{(i)}(t)$ is governed by a system of (non-autonomous) ODEs, obtained by using the fact that $\mathbb E[x_i(t)=1]=p_1^{(i)}(t)$. By replacing the vector $X(t)$ with the vector $p_1(t):=[p_1^{(1)}(t),\dots,p_1^{(n)}(t)]$ in \eqref{eq:Q}, and by using the Chapman-Kolmogorov equation~\cite{levin2006book}, we obtain $[\dot p_0^{(i)}\,\dot p_1^{(i)}]=[p_0^{(i)}\,p_1^{(i)}]Q_i(p_1(t),\varepsilon(t))$, which yields
  \begin{equation}\label{eq:meanfield}
  \begin{alignedat}{1}
     \dot p_0^{(i)} &= - \rho_{01}^{(i)}(p_1(t),\varepsilon(t))  p_0^{(i)}+ \rho_{10}^{(i)}(p_1(t)) p_1^{(i)}\,,\\
     \dot p_1^{(i)} &= \rho_{01}^{(i)}(p_1(t),\varepsilon(t)) p_0^{(i)}-\rho_{10}^{(i)}(p_1(t)) p_1^{(i)}\,,
     \end{alignedat}
   \end{equation}
 for any $i\in \mathcal{V}$. Note that system \eqref{eq:meanfield} is non-autonomous due to the dependency of  $\rho_{01}^{(i)}$ on $\varepsilon(t)$ and, ultimately, on $t$, while the other terms depend only on $t$ through the state $p_1(t)$.
   
Despite such a dependency on $\varepsilon(t)$, we can provide a general invariance result for the system in \eqref{eq:meanfield}, provided that $\varepsilon(t)$ is bounded and Lipschitz. Note that if $\varepsilon(t)$ is defined via \eqref{environmentimpact}, then it is necessarily Lipschitz since it is the solution of an ODE. The following lemma shows that, for any function $\varepsilon(t)$ that has these properties, $( p_0^{(i)} \ p_1^{(i)} ) $ is well-defined as a probability vector for all $t \in \mathbb{R}_{\ge 0}$ and for all $i\in \mathcal{V}$.

\begin{lemma}\label{lemma1} Assume that $\varepsilon(t)$ is  Lipschitz-continuous and bounded for any $t\in \mathbb R_{\geq 0}$. Then, for all $i \in \mathcal{V}$, the set
$\{ ( p_0^{(i)} \ p_1^{(i)} ) : p_0^{(i)}, p_1^{(i)}\ge 0, \ p_0^{(i)}+ p_1^{(i)} = 1 \} $ is positive invariant under \eqref{eq:meanfield}.
\end{lemma}
\begin{proof}
Consider any $i \in \mathcal{V}$. First, we observe that $\dot p_0^{(i)} + \dot p_1^{(i)} = 0$, so $p_0^{(i)}+ p_1^{(i)} = 1$ for all $t \in \mathbb{R}_{\ge 0}$. Then, since the vector field in \eqref{eq:meanfield} is Lipschitz-continuous, Nagumo’s Theorem can be applied~\cite{blanchini1999set}. We are left to check the value of the field at the boundary of the domain. Note that $\dot p_0^{(i)} \ge 0$ if $p_0^{(i)} = 0$ and $\dot p_1^{(i)} \ge 0$ if $p_1^{(i)} = 0$, so $p_0^{(i)}, p_1^{(i)}\ge 0$ for all $t \in \mathbb{R}_{\ge 0}$. 
\end{proof}

It follows directly from Lemma~\ref{lemma1} that, for any $i \in \mathcal{V}$, only one of the two equations in \eqref{eq:meanfield} is sufficient to describe the behavioral evolution of individual $i$. Hence, the mean-field dynamics of the population behavior ultimately consist of an $n$-dimensional set of non-autonomous ODEs. 

Next, let us define the average probability for a randomly selected individual to act responsibly at time $t$,
\begin{equation}\label{macro}
 x(t):= \dfrac{1}{n}\sum_{i \in \mathcal{V}} p_1^{(i)}(t)\,.
\end{equation}
Let $\bar{x}_1(t):= \tfrac{1}{n}\big| \{i \in \mathcal{V} \ : \ x_i(t)=1 \} \big|$ denote the fraction of individuals who behave responsibly at time $t$. For large-scale populations, the fraction $\bar{x}_1(t)$ can be approximated by the macroscopic variable $x(t)$ with arbitrary accuracy (while the two quantities coincide in the limit $n\to\infty$) for any finite time horizon~\cite{limittheo,zino2}, allowing us to accurately study the population behavior from a macroscopic perspective. Finally, in the mean-field approach, the temporal evolution of the whole system is captured by the coupling between i) the system of $n$ independent ODEs defined by \eqref{eq:meanfield} that governs the behavioral evolution of all individuals in the population, and ii) the mean-field dynamics of the environmental impact, which is obtained by replacing the term $\bar x_0(t)$ in \eqref{eq:rate} with the corresponding macroscopic variable $1-x(t)$. Such a coupling reads
\begin{alignat}{1}\label{eq:meanfield_system}
         \dot p_1^{(i)}& = \rho_{01}^{(i)}(p_1(t),\varepsilon(t)) (1-p_1^{(i)})-\rho_{10}^{(i)}(p_1(t)) p_1^{(i)}\,,\forall i\in\mathcal V,\notag\\
  \dot\varepsilon \hspace{6pt} &=\Big(\gamma \Big[1-\dfrac{1}{n}\sum_{i \in \mathcal{V}} p_1^{(i)}(t)\Big]-\tau\Big)\varepsilon,
\end{alignat}
which is an autonomous system of $n+1$ ordinary differential equations.

In the rest of this paper, we study the system in \eqref{eq:meanfield_system} under the following simplifying assumption, to allow for the theoretical analysis of the model.

\begin{assumption}\label{assumptionsimple}
For any $i \in \mathcal{V}$, we assume that $i$ is influenced by the entire population, i.e., $\mathcal N_i=\mathcal V$ for all $i \in \mathcal{V}$.
\end{assumption}
Under Assumption \ref{assumptionsimple}, the incentive functions in \eqref{incentive} reduce to
\begin{equation}\label{meanfieldincentives}
\begin{alignedat}{1}
\iota_1 (x(t),\varepsilon(t)) &:=x(t) + \mu \varepsilon(t) + \alpha\,, \\
    \iota_0 (x(t))& :=1-x(t) + \kappa - \sigma\,,
  \end{alignedat}
\end{equation}
where the index $i$ was discarded, as the incentive functions are no longer individual-dependent and the dependency on the state of other nodes reduces to a dependency on the macroscopic variable $x(t)$. 

Before analyzing the obtained system, we will now make some realistic assumptions on the incentive functions in \eqref{meanfieldincentives} and the mean-field rate of growth/decay $
\bar{r}(x(t)) : = \gamma[1-x(t)]  - \tau$. First, we would like to point out that, currently, none of the negative emission technologies has been demonstrated to be effective at a sufficiently large scale~\cite{rau2019race}. Hence, it is natural to assume that the environmental impact increases if the entire population behaves irresponsibly, i.e., if $x(t) =0$ at time $t$, then $\bar{r}(x(t))>0$. Secondly, concerning the incentive functions, it is reasonable to assume that if there is no environmental impact (i.e., $\varepsilon (t) = 0$), then $ \iota_0 (t) > \iota_1 (t) $ for any $x \in [0,1]$. These two observations directly lead to the following conditions on the parameter $\tau$ and on the cost of responsible behavior $\kappa$. 
\begin{assumption}\label{as:cost}
We make the assumption that: (i) $\tau<\gamma$ and (ii) $\kappa > \sigma +\alpha + 1$.
\end{assumption}
Under Assumption~\ref{assumptionsimple}, we derive the macroscopic mean-field evolution of the population in the following proposition. 

\begin{proposition}\label{prop:system}
Under Assumption~\ref{assumptionsimple}, and in the limit of large-scale populations $n\to\infty$, the mean-field evolution of the macroscopic variable $x$ and the environmental impact $\varepsilon$ is governed by the following autonomous planar system of ODEs:
\begin{equation}\label{simplesystem}
  \begin{array}{lll}
    \dot{x} & =& x(1-x) (2x + \mu \varepsilon +\alpha + \sigma - \kappa -1)\,,\\
  \dot{\varepsilon} & =& (\gamma (1-x) - \tau) \varepsilon\,.
\end{array} 
\end{equation}
\end{proposition}

\begin{proof} Note that under Assumption~\ref{assumptionsimple}, \eqref{behaviourevolution} reduces to $\rho_{01}^{(i)}=x(x + \mu \varepsilon + \alpha )$ and $\rho_{10}^{(i)}=(1-x)(1-x + \kappa - \sigma)$ for all $i \in \mathcal{V}$. Using this, system \eqref{simplesystem} immediately follows from \eqref{macro} and \eqref{eq:meanfield_system} by substituting the expressions in \eqref{eq:meanfield_system} in $ \dot{x}(t)= \tfrac{1}{n}\sum_{i \in \mathcal{V}} \dot{p_1}^{(i)}(t)$.
\end{proof}
To show that system \eqref{simplesystem} is well-defined, we will first show that the environmental impact is bounded from above for any initial condition $x(0) \in (0,1)$. 
\begin{lemma}\label{lemma2} Under \eqref{simplesystem}, there exists an $\bar{\varepsilon} \in \mathbb{R}_{>0}$ such that $\varepsilon(t) \le \bar{\varepsilon}$ for all $(x(0), \varepsilon(0)) \in (0,1) \times \mathbb{R}_{\ge 0}$ and for all $t \in \mathbb{R}_{\ge 0}$. 
\end{lemma}
\begin{proof}
Observe that $\dot{\varepsilon}<0$ for any $x>\hat{x}$ and $\varepsilon>0$, where $\hat{x}$ is the solution of $\gamma(1-x) - \tau =0$. Let $\hat{\varepsilon}$ be the solution of $\mu \varepsilon +\alpha + \sigma - \kappa -1 =0$. Then, $\dot{x} >0$ for any $x \in (0,1)$ and $\varepsilon>\hat{\varepsilon}$. We will now use proof by contradiction to show that $\varepsilon(t)$ is bounded for all $t \in \mathbb{R}_{\ge 0}$ and for all $(x(0), \varepsilon(0)) \in (0,1) \times \mathbb{R}_{\ge 0}$. Assume that $\varepsilon(t)$ is unbounded. Then, for any $M>0$ there exists a time $t$ such that $\varepsilon(t)>M$. Let us consider a time $\bar{t}$ such that $\varepsilon(\bar{t})> \hat{\varepsilon}$ and $x(\bar{t}) \in (0,1)$, so $\dot{x}(\bar{t})>0$. Note that $\dot{x}( \bar{t})  \ge q(1-x),$ for some constant $q>0$ and $x \in (0,1)$. By the Gr\"{o}nwall-Bellman inequality~\cite{pachpatte1997inequalities}, we find that $x(t) \ge x(\bar{t}) e^{q(t-\bar{t})}$ for all $t \ge \bar{t}$. Hence, there exists a $ t^{\ast}$ such that $x(\bar{t}+ t^{\ast}) > \hat{x}$, so $\dot{\varepsilon}(\bar{t}+ t^{\ast})<0$. Note that for any $x \in [0,1]$ and $\varepsilon \in \mathbb{R}_{> 0}$, we have $\dot{\varepsilon} < \gamma \varepsilon$, so the Gr\"{o}nwall-Bellman inequality yields $\varepsilon (t) < \varepsilon(\bar{t})e^{\gamma(t-\bar{t})}$ for all $t \ge \bar{t}$. Thus, $\varepsilon (\bar{t}+ t^{\ast}) < \varepsilon(\bar{t})e^{\gamma t^{\ast}}$. Since $\dot{\varepsilon}(\bar{t}+ t^{\ast})<0$, there does not exist a time $t$ such that $\varepsilon(t)> M $ for any $M> \varepsilon(\bar{t})e^{\gamma t^{\ast}}$. 
\end{proof}
Using the above result, the following lemma shows that \eqref{simplesystem} is well-defined for all $t \in \mathbb{R}_{\ge 0}$. 
\begin{lemma}\label{lemma3}  The set $(x,\varepsilon) \in [0,1]\times \mathbb{R}_{\ge 0} $ is positive invariant under \eqref{simplesystem}.
\end{lemma}

\begin{proof}
It follows directly from Lemma~\ref{lemma1} and \ref{lemma2} that for any initial condition with $x(0) \in (0,1)$, $(x,\varepsilon) \in [0,1]\times \mathbb{R}_{\ge 0} $ for all $t \in \mathbb{R}_{\ge 0}$. It remains to analyze the behavior of the system on the boundaries. For $x(0)=0$, the dynamics reduce to $\dot x=0$ and $\dot\varepsilon=(\gamma-\tau)\varepsilon$. Hence, the solution of the system is given by $x(t) = 0$ and $\varepsilon(t)=\varepsilon(0)e^{(\gamma-\tau)t}$, which is in the set for all $t \in \mathbb{R}_{\ge 0}$. For $x(0)=1$, the dynamics reduce to $\dot x=0$ and $\dot\varepsilon=-\tau\varepsilon$, which converges exponentially to the origin, belonging to the set. 
\end{proof}

\section{Main Results}\label{sectionresults}

\noindent In this section, we perform an analysis of the planar mean-field system in \eqref{simplesystem} to fully unveil its asymptotic behavior. We start by characterizing the equilibria of \eqref{simplesystem} and establishing their local stability properties, which are presented in the following proposition. 

\begin{proposition}\label{propequi}
Under Assumption~\ref{as:cost}, the system in \eqref{simplesystem} has three equilibria:
\begin{enumerate}
    \item[i)] $(x, \varepsilon) = (0,0)$, which is a saddle point;
    \item[ii)] $(x, \varepsilon) =(1,0)$, which is a saddle point;
    \item[iii)] $(x, \varepsilon) =(1- \tfrac{\tau}{\gamma},\tfrac{1}{\mu} [\tfrac{2 \tau}{\gamma} + \kappa -\sigma - \alpha -1])$, which is an unstable spiral.
  \end{enumerate}
\end{proposition}
\begin{proof}
Let Assumption~\ref{as:cost} hold. Solving $ \dot{x} =0$ yields $x=0$, $x=1$ or $x = \tfrac{1}{2}(-\mu \varepsilon - \alpha - \sigma + \kappa +1)$. If $x=0$ or $x=1$, then the only solution to $\dot{\varepsilon} =0$ is $\varepsilon = 0$, giving equilibria $(x, \varepsilon) = (0,0)$ and $(x, \varepsilon) = (1,0)$. 

Now consider $x = \tfrac{1}{2}(-\mu \varepsilon - \alpha - \sigma + \kappa +1)$. By solving $0= \dot{\varepsilon} =(\gamma (1-x) - \tau) \varepsilon$, we find 
equilibrium 
\begin{equation}\label{equili}
    (x, \varepsilon) =(1- \tfrac{\tau}{\gamma},\tfrac{1}{\mu} [\tfrac{2 \tau}{\gamma} + \kappa -\sigma - \alpha -1]).
\end{equation}
Note that $\varepsilon=0$ is not an option for $x\in [0,1]$, as this gives $x = \tfrac{1}{2}( - \alpha - \sigma + \kappa +1)> 1$, by Assumption~\ref{as:cost}(ii). 

Next, we examine the local stability. First, consider the equilibrium~\eqref{equili}. Linearizing \eqref{simplesystem} around this equilibrium is equivalent to linearizing the system $(\dot{\tilde{x}}, \dot{\tilde{\varepsilon}})$ around the origin, with 
\begin{align*}
    \tilde{x} & := x-1 + \tfrac{\tau}{\gamma}\,, \\
    \tilde{\varepsilon} & := \varepsilon -\tfrac{1}{\mu} [\tfrac{2 \tau}{\gamma} + \kappa -\sigma - \alpha -1]\,. 
\end{align*}
Doing so yields
\begin{equation*}
  \begin{bmatrix}\dot{\tilde{x}} \\ \dot{\tilde{\varepsilon}} \end{bmatrix} = \begin{bmatrix} 2 \tfrac{\tau}{\gamma} (1-\tfrac{\tau}{\gamma})& \mu \tfrac{\tau}{\gamma} (1-\tfrac{\tau}{\gamma}) \\ -\tfrac{1}{\mu}(2 \tau + \gamma[\kappa - \sigma - \alpha -1  ]) & 0\end{bmatrix} \begin{bmatrix}{\tilde{x}} \\ {\tilde{\varepsilon}} \end{bmatrix}, 
\end{equation*}
where the Jacobian matrix has eigenvalues 
\begin{align}\label{eigenvalue}
  \lambda_{\pm} = & \tfrac{\tau}{\gamma} (1-\tfrac{\tau}{\gamma}) \pm \sqrt{ \tfrac{\tau^2}{\gamma^2} (1-\tfrac{\tau}{\gamma})^2 - \tfrac{\tau}{\gamma} (1-\tfrac{\tau}{\gamma}) (2\tau + \gamma [\kappa -\sigma - \alpha -1])}.
\end{align}
Note that the radicand is negative iff $\tau^2 + \gamma (2\gamma -1) \tau + \gamma^3[\kappa -\sigma - \alpha -1]>0$. The equation $\tau^2 + \gamma (2\gamma -1) \tau + \gamma^3[\kappa -\sigma - \alpha -1] = 0$ is solved by $$\tau_{\pm} = -\tfrac{1}{2}\gamma(2 \gamma -1) \pm \tfrac{1}{2}\sqrt{\gamma^2(2 \gamma -1)^2 - 4 \gamma^3[\kappa -\sigma - \alpha -1]}.$$ Observe that $\tau_{\pm} \notin \mathbb{R}_{>0}$, due to Assumption~\ref{as:cost}(ii), so $\tau^2 + \gamma (2\gamma -1) \tau + \gamma^3[\kappa -\sigma - \alpha -1]>0$ for all $\tau \in \mathbb{R}_{>0}$. Thus, the radicand in \eqref{eigenvalue} is negative and $\text{Re}( \lambda_{+}) = \text{Re}( \lambda_{-}) = \tfrac{\tau}{\gamma} (1-\tfrac{\tau}{\gamma})> 0 $ by Assumption~\ref{as:cost}(i), implying that the equilibrium under consideration is an unstable spiral.  

Next, consider $(x, \varepsilon) =(0,0)$. Linearizing the system in \eqref{simplesystem} around $(x, \varepsilon) =(0,0)$ yields a Jacobian matrix with eigenvalues $\gamma - \tau > 0$ and $\alpha + \sigma -1 - \kappa$. By Assumption~\ref{as:cost}(ii), $\alpha + \sigma -1 - \kappa < -2 < 0$, so $(x, \varepsilon) =(0,0)$ is a saddle point. 

Finally, we consider $(x, \varepsilon) =(1,0)$. Linearizing \eqref{simplesystem} around $(x, \varepsilon) =(1,0)$ gives a Jacobian matrix with eigenvalues $- \tau < 0$ and $\kappa -(\sigma + \alpha +1) > 0$ (by Assumption~\ref{as:cost}(ii)), so $(x, \varepsilon) =(1,0)$ is a saddle point.
\end{proof}

Employing the above local stability properties of the system equilibria, we derive the following (almost) global convergence result. 
\begin{theorem}\label{theo1} Let Assumption~\ref{as:cost} hold. If the initial condition $(x(0),\varepsilon(0))$ is in the interior of the domain $[0,1]\times \mathbb{R}_{\ge 0}$ and does not coincide with the interior equilibrium in \eqref{equili}, then all solutions of the system in \eqref{simplesystem} converge to a limit cycle.
\end{theorem}
\begin{proof}
Let Assumption~\ref{as:cost} hold, so $\kappa > \sigma +\alpha + 1$. To prove that the system in \eqref{simplesystem} converges to a limit cycle, we first need to study the behavior of \eqref{simplesystem} close to the boundary of its domain. Consider the boundary $x=1$. Let us assume that there exists a trajectory that reaches $x = 1- \epsilon$ at a time $t_0$, where $\epsilon \in (0, \tfrac{\tau}{\gamma})$ is arbitrarily infinitesimally small. Note that for $x\in [1- \epsilon,1)$ and $\varepsilon  \le \tfrac{1}{\mu} (\kappa - \sigma -\alpha -1)$, we have $2x + \mu \varepsilon +\alpha + \sigma - \kappa -1 < 1 + \mu \varepsilon +\alpha + \sigma - \kappa  \le 0$, which implies that $$\dot{x} = x(1-x) (2x + \mu \varepsilon +\alpha + \sigma - \kappa -1)<0.$$ Hence, $\dot{x}$ can only be positive for $\varepsilon > \tfrac{1}{\mu} (\kappa - \sigma -\alpha -1)$.
 Let us consider any trajectory with $x(t_0)=1-\epsilon$ that enters the region $$\mathcal{R} : = [1- \epsilon, 1] \times \big( \tfrac{1}{\mu} (\kappa - \sigma -\alpha -1 ), \varepsilon(t_0) \big].$$ We will now show that it is not possible for a trajectory in $\mathcal{R}$ to reach the boundary $x=1$. Note first that for $\varepsilon>0$ and $x> 1-\tfrac{\tau}{\gamma}$, we have $\dot{\varepsilon} = (\gamma (1-x) - \tau) \varepsilon < 0$, so $\dot{\varepsilon}<0 $ for $(x,\varepsilon) \in [1-\epsilon,1]\times \mathbb{R}_{>0}$ and the trajectory cannot exit $\mathcal{R}$ from the the top. Next, let us define $u(t) = 1-x(t)$.
Since $\varepsilon(t)\le \varepsilon(t_0)$ for any $t\ge t_0$, it follows that $\dot{x}  \le p(1-x),$ for some constant $p>0$, which is equivalent to $-\dot{u} \le -p(-u)$. 
 By the Gr\"{o}nwall-Bellman inequality~\cite{pachpatte1997inequalities}, we have $-u(t) \le -u(t_0)e^{-p(t-t_0)}$, or equivalently, $$x(t) \le 1-\epsilon e^{-p (t-t_0)}<1,$$ for any $t\geq t_0$. Next, note that in $\mathcal{R}$, $\dot{\varepsilon} \le - ( \tau- \epsilon\gamma ) \varepsilon < - \tfrac{1}{\mu}( \tau- \epsilon\gamma ) (\kappa - \sigma -\alpha -1 ).$ Thus, $|\dot{\varepsilon}| > \tfrac{1}{\mu}( \tau- \epsilon\gamma ) (\kappa - \sigma -\alpha -1  ) $. The length of the $\varepsilon$-axis in $\mathcal{R}$ is less than $ \varepsilon(t_0) -\tfrac{1}{\mu} (\kappa - \sigma -\alpha -1) $. Hence, there exists a $$\tilde{t} \le \dfrac{\mu \varepsilon(t_0) - (\kappa - \sigma -\alpha -1 )}{( \tau- \epsilon\gamma ) (\kappa - \sigma -\alpha -1 )}$$ such that $\varepsilon(t_0 + \tilde{t}) \le \tfrac{1}{\mu} (\kappa - \sigma -\alpha -1 ) $. At time $t_0+\tilde{t}$, we have $x(t_0+ \tilde{t}) \le 1-\epsilon e^{-p (\tilde{t})} <1$, and the trajectory is in the region $\mathcal{S} : = [1- \epsilon, 1] \times \big[0, \tfrac{1}{\mu} [\kappa - \sigma -\alpha -1 ]\big].$ Since $\dot{x}<0$ for $\varepsilon  \le \tfrac{1}{\mu} (\kappa - \sigma -\alpha -1)$, the trajectory will move away from the boundary and cannot reach $x=1$. 
 
 Similarly, we can show that any trajectory starting in the interior cannot reach the boundaries $\varepsilon = 0$ and $x=0$. This implies that it is impossible to reach the boundary equilibria if the initial conditions of the system are in the interior of $[0,1] \times \mathbb{R}_{\ge 0}$. 
 
 Lastly, consider the set $(x,\varepsilon) \in (0,1) \times \mathbb{R}_{> 0}$. Since the unique equilibrium in the interior is unstable, there does not exist a homoclinic orbit. Moreover, Lemma~\ref{lemma2} guarantees that all solutions are bounded. Hence, by the generalized Poincar\'{e}-Bendixson theorem~\cite{teschl2012ordinary}, every non-empty compact $\omega$-limit set of an orbit is periodic. 
 \end{proof}

\begin{figure*}
\centering

\includegraphics[width=0.99\textwidth]{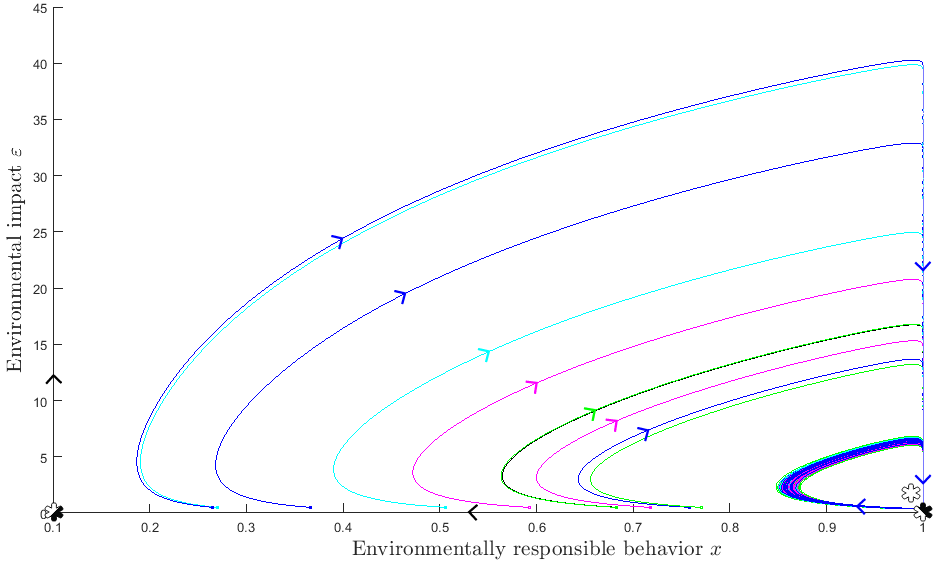}
\caption{Simulated trajectories of the system in \eqref{simplesystem} for parameter values $\alpha = 0.3$, 
$\sigma = 0.6$, 
$\kappa=3$, 
$\gamma= 10$, 
$\tau = 0.1$ and
$\mu = 0.6$. Saddle points and unstable equilibria are marked with black-white and white asterisks, respectively.\label{figtraject}}
\end{figure*}

For the sake of completeness, we report here a brief characterization of the behavior of the system if the initial condition is on the boundary of the domain. 
\begin{proposition}
Under Assumption~\ref{as:cost}, the following properties hold:
\begin{enumerate}
\item[i)] If $x(0)=1$ and $\varepsilon(0)\ge 0$, then the solution of \eqref{simplesystem} converges to the equilibrium $(1,0)$;
\item[ii)] If $x(0)<1$ and $\varepsilon(0)=0$, then the solution of \eqref{simplesystem} converges to the equilibrium $(0,0)$;
\item[iii)] If $x(0)=0$ and $\varepsilon(0)>0$, then the solution of \eqref{simplesystem} diverges toward $(0,\infty)$.
\end{enumerate}
\end{proposition}
\noindent Theorem~\ref{theo1} guarantees that almost all of the trajectories that start in the interior of the domain converge to a limit cycle. This finding is illustrated by a set of simulations, reported in Fig.~\ref{figtraject}, where it is shown that all of the simulated trajectories converge to a periodic solution. Additionally, Fig.~\ref{figtraject} suggests that before the trajectory reaches the natural oscillations in the limit cycle, the environmental impact might increase to an alarmingly high level during the transient phase, dependent on the initial system conditions. Such an observation calls for the design of optimal control strategies to influence the system trajectory in the transient regime.

\begin{figure}
\centering
\includegraphics[width=0.6\linewidth]{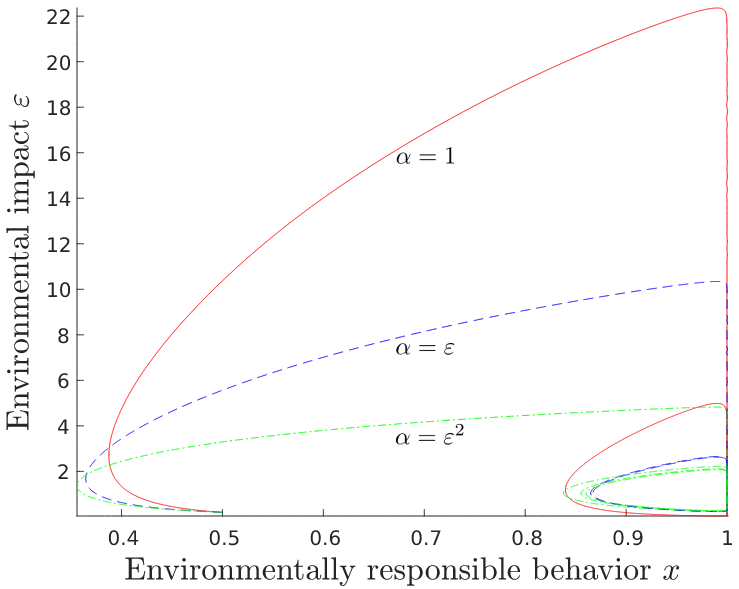}
\caption{Simulated trajectories of the system in \eqref{simplesystem} for different choices of the control function $\alpha$. Common parameter values are
$\sigma = 0.6$, 
$\kappa=3$, 
$\gamma= 10$, 
$\tau = 0.1$ and
$\mu = 0.6$. \label{control}} 
\end{figure}

In particular, we believe that optimal control strategies in terms of environmental subsidies and awareness campaigns may be designed by letting the parameters $\kappa$ and $\mathcal \alpha$ be time-varying control parameters. Specifically, the use of feedback control schemes, where we let $\alpha$ and $\kappa$ depend on $t$ through $\varepsilon(t)$, might be extremely beneficial toward mitigating the increase in environmental impact during the transient phase. This intuition is bolstered by simulations of the proposed control scheme in Fig.~\ref{control}: implementing a control action that is linearly (or even super-linearly) proportional to the environmental impact seems to be highly beneficial in reducing its peaks. The analysis of such feedback-controlled policies might be performed using arguments similar to those used in the proof of Lemma~\ref{lemma2}, through which one can estimate an upper bound on the peak of $\varepsilon(t)$ throughout the entire trajectory. Therefore, one of our future research objectives is to design optimal control strategies, potentially in feedback with the system, to guarantee that $\varepsilon(t)$ is always less than the critical threshold above which the planet becomes unsuitable for life.


\section{Conclusion}\label{sec:con}

\noindent In this paper, we proposed a novel stochastic network model that captures the coevolution of human behavior concerning environment-related issues and environmental impact. Our modeling framework includes a variety of factors such as policy interventions, negative emission technologies, social influence, a behavioral response to increases in environmental impact, and the cost of environmentally-friendly behavior. By employing a mean-field approach, we derived a deterministic approximation of the system in the limit of large-scale populations, for which we performed a complete asymptotic analysis. Specifically, we proved global convergence to a periodic solution for almost all initial conditions. 

Our modeling framework and results open up the path for several directions of future research. First, our theoretical results are derived under the simplifying Assumption~\ref{assumptionsimple} of an all-to-all network of interactions. To better approximate real-world scenarios, the behavioral-environmental feedback model can be studied by employing non-trivial social networks. Second, we assumed a linear behavioral response to the environmental impact, but more complex nonlinear functions may be considered. In particular, one may consider extending the framework to a multi-population scenario with cautious and reckless subpopulations, modeled by assigning different functions for the environmental response. By including a degree of homophily, i.e., a tendency of people to interact with like-minded individuals, one can explore the role of a polarized network structure in the evolution of the environmental population behavior. 

Third, in our original model formulation, we assumed that the policy interventions, i.e., environmental subsidies and awareness campaigns, are constant over time. As we discussed through the numerical simulations in Fig.~\ref{control}, future efforts should be placed on investigating the possibility to mitigate extreme oscillations of the system via time-varying control policies and, in particular, state-dependent policies, where the effort placed by public authorities is defined as a feedback function of the state of the environment.

\bibliographystyle{ieeetr}
\bibliography{bib}

\end{document}